\documentclass[a4paper, reqno,11pt,notitlepage,oneside]{amsart}
    \usepackage[utf8]{inputenc}
    \usepackage[textwidth=360pt,textheight=615pt]{geometry}

    \usepackage{amsmath,amssymb,amsthm,graphicx,color,eucal,stmaryrd,scalerel}
    \usepackage[shortlabels]{enumitem} 
    \usepackage[all,cmtip]{xy}
    \usepackage[T1]{fontenc}
    
    \usepackage{hyperref} 
    \definecolor{darkblue}{rgb}{0,0,.85} 
    \definecolor{darkred}{rgb}{0.84,0,0}
    \hypersetup{colorlinks = true,
            linkcolor = darkblue,
            urlcolor  = darkblue,
            citecolor = darkred,
            anchorcolor = darkblue}

    \SetLabelAlign{LeftAlignWithIndent}{\hspace*{2.0ex}\makebox[1.5em][l]{#1}}
    
    \makeatletter
    \def\paragraph{\@startsection{paragraph}{4}%
    \z@\z@{-\fontdimen2\font}%
    {\normalfont\bfseries}}
    \makeatother
    
    \numberwithin{equation}{subsection}
        	
    \newdir{ >}{{}*!/-5pt/@{>}}
    \SelectTips{cm}{10}

    \newtheorem{lem}{Lemma}[subsection]
    \newtheorem{cor}[lem]{Corollary}
    \newtheorem{thm}[lem]{Theorem}
    \newtheorem{prop}[lem]{Proposition}
    \newtheorem{thmi}{Theorem} 

    \theoremstyle{definition}
    \newtheorem{definition}[lem]{Definition}
    \newtheorem{construction}[lem]{Construction}
    \newtheorem{rem}[lem]{Remark}
    \newtheorem{example}[lem]{Example}
    
    \newcommand{\mf}[1]{\mathfrak{#1}}
    
    \newcommand{\mc}[1]{\mathcal{#1}}
    \newcommand{\mb}[1]{\mathbf{#1}}
    \newcommand{\msf}[1]{\mathsf{#1}}
    \newcommand{\ov}[1]{\overline{#1}}

    \newcommand{\op}{\operatorname}
    \DeclareMathOperator*{\colim}{colim}
    \DeclareMathOperator{\Hom}{Hom}
    \DeclareMathOperator{\Aut}{Aut}

    \DeclareMathOperator{\Spa}{Spa}
    \DeclareMathOperator{\interior}{int}
    
    \newcommand{\oc}{\mathrm{oc}}
    \newcommand{\et}{\mathrm{\acute{e}t}}
    \newcommand{\adm}{\mathrm{adm}}
    
    \newcommand{\proet}{\mathrm{pro\acute{e}t}}
    \newcommand{\profet}{\mathrm{prof\et}}
    \newcommand{\rmta}{\mathrm{ta}}
    \newcommand{\dJ}{\mathrm{dJ}}
    \newcommand{\alg}{\mathrm{alg}}
    \newcommand{\sep}{\mathrm{sep}}
    \newcommand{\cnts}{\mathrm{cnts}}
    \newcommand{\smallto}{\scaleto{(\to)}{5.5pt}}
    
    \newcommand{\cat}[1]{\operatorname{\mathbf{#1}}} 
    
    \newcommand{\GSet}[1]{{#1\text{-}\cat{Set}}}
    \newcommand{\pfset}[1]{{#1\text{-}\cat{pFSet}}}
    \newcommand{\Cov}{\cat{Cov}}
    \newcommand{\UCov}{\cat{UCov}}
    \newcommand{\Et}{\cat{\acute{E}t}}
    \newcommand{\FEt}{\cat{F\acute{E}t}}
    \newcommand{\UFEt}{\cat{UF\acute{E}t}}
    \newcommand{\ProEt}{\cat{Pro\acute{E}t}}
    \newcommand{\Sh}{\cat{Sh}}
    \newcommand{\Loc}{\cat{Loc}}

    \newcommand{\calF}{\mathcal{F}}
    
    \newcommand{\cO}{\mathcal{O}}
    \newcommand*\isomto{%
        \xrightarrow{\raisebox{-0.2 em}{\smash{\ensuremath{\sim}}}}%
    }
    \newcommand{\stacks}[1]{\cite[\href{https://stacks.math.columbia.edu/tag/#1}{Tag~#1}]{StacksProject}}
	
	\newcounter{steps}

    \title{Variants of the de Jong fundamental group} 
    \date{\today}
    
   \author{Piotr Achinger}
    \address{Institute of Mathematics of the Polish Academy of Sciences \newline \indent ul.\ Śniadeckich 8, 00-656 Warsaw, Poland}
    \email{pachinger@impan.pl}

    \author{Marcin Lara}
    \address{Institute of Mathematics of the Polish Academy of Sciences \newline \indent ul.\ Śniadeckich 8, 00-656 Warsaw, Poland}
    \email{marcin.lara@impan.pl}

    \author{Alex Youcis}
    \address{\begin{itemize} 
    \item[(1)] Institute of Mathematics of the Polish Academy of Sciences \newline \indent ul.\ Śniadeckich 8, 00-656 Warsaw, Poland
    \item[(2)] Graduate School of Mathematical Sciences, The University of Tokyo,
    3-8-1 Komaba, Meguro-ku, Tokyo, 153-8914, Japan
    \end{itemize}}
     \email{ayoucis@ms.u-tokyo.ac.jp}

\begin{document}

\begin{abstract}
For a rigid space $X$, we answer two questions of de~Jong about the category $\mathbf{Cov}^\mathrm{adm}_X$ of coverings which are locally in the admissible topology on $X$ the disjoint union of finite \'etale coverings: we show that this class is different from the one used by de~Jong, but still gives a tame infinite Galois category. In addition, we prove that the objects of $\mathbf{Cov}^\mathrm{et}_X$ (with the analogous definition) correspond precisely to locally constant sheaves for the pro-\'etale topology defined by Scholze.
\end{abstract}

\maketitle

\section{Introduction}

Let $K$ be a non-archimedean field. By a rigid $K$-space we shall mean an adic space locally of finite type over $K$. In \cite{ALY1P1}, we introduced a new definition of covering space in rigid-analytic geometry, called \emph{geometric coverings}, and showed that the category $\Cov_X$ of geometric coverings of a rigid $K$-space $X$ gives rise to a fundamental group $\pi^{\rm ga}_1(X, \ov{X})$, called the \emph{geometric arc fundamental group}. In this paper we use these ideas to explicate two previously studied objects: de Jong's notion of covering spaces, and Scholze's theory of pro-etale local systems.

\medskip

\paragraph{Extensions of de Jong covering spaces} The notion of `covering space' of a rigid $K$-space has been a notoriously difficult concept to fully encapsulate. It was realized early on that, unlike the case of schemes, even relatively well-behaved rigid spaces admit interesting connected covers of infinite degree. The most classical examples of this phenomenon are the Tate uniformization map $\mathbf{G}^\mathrm{an}_{m,K}\to E^\mathrm{an}$, for an elliptic curve $E$ over $K$ with split multiplicative reduction, and the Gross--Hopkins period map
\begin{equation*}
   \pi_\mathrm{GH}\colon \{x\in \mathbf{A}^{1,\mathrm{an}}_{\mathbf{C}_p}: |x|<1\}\to \mathbf{P}^{1,\mathrm{an}}_{\mathbf{C}_p}.
\end{equation*}
Building off ideas of Berkovich, in \cite{deJongFundamental} de Jong sought to define a notion of covering space which could encompass these examples but was robust enough to support a notion of fundamental group.

To define de Jong's coverings (and related variants), let us consider the full subcategories
\[ 
    \Cov^\tau_X (\subseteq \Cov_X) \subseteq \Et_X, \quad \tau\in\{\text{adm}, \text{\'et}, \text{oc}\}
\]
consisting of \'etale maps $Y\to X$ for which there exists a $\tau$-cover $U\to X$ such that $Y_U\to U$ is the disjoint union of finite \'etale coverings of $U$. Here $\adm,\et,$ or $\oc$ denotes the usual (i.e.\ `admissible'), \'etale, or overconvergent (i.e.\@ partially proper open) Grothendieck topology on $X$, respectively. Every geometric point $\ov{x}$ of $X$ furnishes these categories with a fiber functor 
\begin{equation*}
    F_{\ov{x}}\colon \Cov^\tau_X\to \cat{Set},\qquad F_{\ov{x}}(Y)=\Hom_X(\ov{x},Y).
\end{equation*}
The category $\Cov^{\rm oc}_X$, called the category of \emph{de Jong covering spaces}, is precisely the category considered in \cite{deJongFundamental}. One may think about the category $\Cov^\oc_X$ as being a synthesis of the notion of finite \'etale covering and topological covering (of the Berkovich space associated to $X$). 

The category $\Cov_X^\oc$ is shown in op.\@ cit.\@ to have favorable properties. First, it is shown that the natural infinite degree covering spaces mentioned above are examples of de Jong covering spaces. Significantly deeper, it is shown that if one sets $\UCov^\oc_X$  to be the category of arbitrary disjoint unions of de Jong covering spaces, then the pair $(\UCov^\oc_X,F_{\ov{x}})$ is a \emph{tame infinite Galois category} when $X$ is connected. The notion of a pair $(\mc{C},F)$ being a tame infinite Galois category was developed in \cite{BhattScholze} as a generalization of the classical theory of Galois categories. In particular, there is a topological group $\pi_1(\mc{C},F)$, called the \emph{fundamental group} of $(\mc{C},F)$, such that
\begin{equation*}
    F\colon \mc{C}\isomto \GSet{\pi_1(\mc{C},F)}
\end{equation*}
is an equivalence where $\GSet{\pi_1(\mc{C},F)}$ is the category of sets endowed with a continuous action of $\pi_1(\mc{C},F)$. The fundamental group of the pair $(\UCov^\oc_X,F_{\ov{x}})$ is called the \emph{de Jong fundamental group} and denoted $\pi_1^\dJ(X,\ov{x})$. 

Despite these positive aspects of the category $\Cov_X^\oc$, there is an obvious downside. Namely, it is not obvious whether the notion of a de Jong covering space is local on the target for the admissible topology. For this reason, in \cite{deJongFundamental} the following two questions are posed (using different language):
\begin{itemize}[leftmargin=1cm]
    \item Does the equality $\Cov^{\rm oc}_X=\Cov^{\rm adm}_X$ hold?
    \item If not, is $(\UCov^\adm_X,F_{\ov{x}})$ a tame infinite Galois category? 
\end{itemize}
We give a negative answer to the first question using an explicit construction relying on a careful analysis of Artin–Schreier coverings.

\begin{thmi}[{See Proposition \ref{prop:counterex}}]\label{intro-thm:counterexample}
    Let $K$ be a non-archimedean field of characteristic $p$ and let $X$ be an affinoid annulus over $K$. Then, the containment $\Cov_X^\oc\subseteq \Cov_X^\adm$ is strict.
\end{thmi}

One may think the existence of such an example is a subtlety related to characteristic $p$ geometry. But, in \cite{Gaulhiac} Gaulhiac has cleverly adapted our construction to produce an analogous example in mixed characteristic. Note that we do not expect such examples to exist when $K$ is of equicharacteristic $0$. In fact, in \cite[Corollary 4.17]{ALY2}, we show that quite often the equality $\Cov_X^\oc=\Cov_X^\et$ holds in equicharacteristic 0.

In \cite{ALY1P1} we show that the pair $(\Cov_X,F_{\ov{x}})$ is a tame infinite Galois category. As the notion of geometric covering is \'etale local on the target, is closed under disjoint unions, and contains finite \'etale coverings we see that $\UCov^\tau_X$, with the obvious definition, is contained in $\Cov_X$. Combining these two results we answer de Jong's second question.

\begin{thmi}[See Theorem~\ref{taus are tame}]\label{thmi:taus are tame}
    Let $X$ be a connected rigid $K$-space with geometric point $\ov{x}$. For every $\tau\in\{\adm,\et,\oc\}$, the pair $(\UCov^\tau_X, F_{\ov x})$ is a tame infinite Galois category.
\end{thmi}

\noindent 
Consequently, there is a topological group $\pi_1^{\dJ,\tau}(X,\ov{x})$, which we call the \emph{$\tau$-adapted de Jong fundamental group}, and an equivalence of categories
\[ 
    F_{\ov x} \colon \UCov_X^\tau \isomto \GSet{\pi_1^{\dJ,\tau}(X,\ov{x})}.
\]
As $\pi_1^{\dJ,\oc}(X,\ov{x})=\pi_1^\dJ(X,\ov{x})$ this notion extends that of the de Jong fundamental group.

\paragraph{Local systems for the pro-\'etale topology} In the second part of this paper, we show the largest of these three categories $\Cov^\et_X$ is not just of purely theoretical interest and connects to previously studied objects. Recall that in \cite{Scholzepadic}, Scholze introduced a site $X_\proet$ for a rigid $K$-space $X$, called there the \emph{pro-\'etale topology}. Its covers are roughly an \'etale cover of $X$ followed by an inverse limit of finite \'etale covers. 

As local systems are usually analyzed using fundamental group techniques, it natural to ask the following question: if $X$ is connected, is the category $\Loc(X_\proet)$ of locally constant sheaves of sets for the pro-\'etale topology, endowed with a fiber functor induced by a geometric point $\ov x$, a tame infinite Galois category? 

\begin{thmi}[See Theorem \ref{main locally constant theorem}]
    The functor associating to a geometric covering $Y\to X$ the corresponding sheaf on the pro-\'etale site $X_\proet$ induces an equivalence of categories 
    \[ 
        \Cov^\et_X \isomto \cat{Loc}(X_\proet).
    \]
\end{thmi}

Considering Theorem \ref{thmi:taus are tame}, we may give a precise answer to our question. While $(\Loc(X_\proet),F_{\ov{x}})$ is not a tame infinite Galois category, as it is not closed under coproducts, if we set $\cat{ULoc}(X_\proet)$ to be the category of disjoint unions of objects of $\Loc(X_\proet)$ then the pair $(\cat{ULoc}(X_\proet),F_{\ov{x}})$ is a tame infinite Galois category. Moreover, the fundamental group of $(\cat{ULoc}(X_\proet),F_{\ov{x}})$ is identified with $\pi_1^{\dJ,\et}(X,\ov{x})$ and so one has an equivalence
\begin{equation*}
    F_{\ov{x}}\colon \cat{ULoc}(X_\proet)\isomto \GSet{\pi_1^{\dJ,\et}(X,\ov{x})}
\end{equation*}

\subsection*{Acknowledgements}
The authors would like to thank Antoine Ducros, Ofer Gabber, David Hansen, Johan de~Jong, Emmanuel Lepage, Shizhang Li, Jérôme Poineau, and Peter Scholze for useful conversations and comments on the preliminary draft of this manuscript. This work is a part of the project KAPIBARA supported by the funding from the European Research Council (ERC) under the European Union’s Horizon 2020 research and innovation programme (grant agreement No 802787). 

\subsection*{Notation and conventions} Throughout this article $K$ shall be a non-archimedean field. We shall use freely the material from \cite{Huberbook} and \cite{FujiwaraKato}. We follow the conventions and notations concerning valuative spaces, adic spaces, and rigid $K$-spaces set out in \cite[\S2-3]{ALY1P1}. In particular, we shall use the universal separated quotient $[X]$ of a taut rigid $K$-space $X$, which coincides with the associated Berkovich space $X^\mathrm{Berk}$.

\section{A non-overconvergent covering space}
\label{example subsection}

In this section we construct an example of a morphism $Y\to X$ that belongs to $\Cov^\adm_X$ but not $\Cov^\oc_X$. For $X$ we take an annulus over a characteristic $p$ non-archimedean field, and $Y \to X$ is constructed by carefully gluing together finite \'etale covers of a neighborhood of the Gauss point which maximally extend over shrinking overconvergent neighborhoods. After that, we establish a result which, in some sense, implies that every element of $\Cov^\adm_X$ not in $\Cov^\oc_X$ must come from such a gluing construction.

\subsection{The example} Let $K$ be a non-archimedean field of characteristic $p>0$ and let $\varpi \in K$ be a pseudouniformizer. Set 
\[
    X=\{|\varpi|\leq |x|\leq |\varpi|^{-1}\} \subseteq \mb{A}^{1, \rm an}_K.
\]
Our example $Y\to X$ is obtained by gluing two families $Y^\pm_n$ of Artin--Schreier coverings of two annuli 
\[
    U^-=\{|\varpi|\leq |x|\leq 1\},\qquad U^+=\{1\leq |x|\leq|\varpi|^{-1}\}
\]
which split over shrinking overconvergent neighborhoods of the intersection
\begin{equation*}
    C=U^-\cap U^+=\{|x|=1\}.
\end{equation*} 

We begin with an analysis of Artin--Schreier coverings of $\mathbf{A}^{1, \rm an}_K$. For a rational number $\alpha$, we define the affinoid opens
\[ 
    D(\alpha)= \{ |x|\leq |\varpi|^{-\alpha} \} 
    \quad \supseteq \quad 
    S(\alpha) = \{|x|=|\varpi|^{-\alpha}\}.
\]
For integers $a,b$ with $b>0$, we denote by $Y_{a,b}$ the following Artin--Schreier covering of $\mathbf{A}^{1, \rm an}_K$:
\[ 
    Y_{a,b} = \left(\op{Spec} K[x,y]/(y^p - y - \varpi^{a}x^{b})\right)^{\rm an}.
\]
It is a $\mathbf{Z}/p\mathbf{Z}$-torsor over $\mathbf{A}^{1, \rm an}_K$, and since $\mathbf{Z}/p\mathbf{Z}$ does not have any non-trivial subgroups, the restriction of $Y_{a,b}$ to an affinoid subdomain $U\subseteq \mathbf{A}^{1, \rm an}_K$ is disconnected if and only if it splits completely over $U$ and if and only if the equation $y^p-y=\varpi^{a}x^b$ has a solution in $\cO(U)$. 

\begin{lem} \label{lem:AS-splitting}
    Let $a, b$ be integers with $b>0$ and let $\alpha$ be a rational number. The following are equivalent: 
    \begin{enumerate}[(a)]
        \item The covering $Y_{a,b}$ splits over $D(\alpha)$.
        \item The covering $Y_{a,b}$ splits over $S(\alpha)$.
        \item We have $\alpha <  a/ b$.
    \end{enumerate}
\end{lem}

\begin{proof}
Write $g=\varpi^{a}x^b$. The unique solution to $y^p - y = -g$ in $K[\![x]\!]$ satisfying $y(0)=0$ is the power series
\[ 
	f = g + g^p + g^{p^2} + \cdots = \sum_{s\geq 1} \varpi^{p^s a} x^{p^s b}.
\]
It converges on $D(\alpha)$ if and only if $\alpha < a/b$, as $\cO(D(\alpha))$ consists of power series $\sum a_n x^n$ with $|a_n|\cdot |\varpi|^{-n\alpha } \to 0$. Therefore, (a) is equivalent to (c). 

Since clearly (a) implies (b), it remains to show (b) implies (a). The ring $\cO(S(\alpha))$ consists of Laurent series $f = \sum_{n\in\mathbf{Z}} a_n x^n$ with $|a_n|\cdot |\varpi|^{-n\alpha}\to 0$ as $|n|\to \infty$. It suffices to show that if $f^p - f \in \cO(D(\alpha))$, then $f\in \cO(D(\alpha))$ (i.e.\ $a_n=0$ for $n<0$). We have $f^p - f$ is equal to $\sum_{n\in \mathbf{Z}} (a_{n/p}^p - a_n)x^n$ where we set $a_{n/p}=0$ if $p$ does not divide $n$. Since $f^p-f\in \cO(D(\alpha))$, we have $a_{n/p}^p - a_n = 0$ for $n<0$. By induction we see that $a_n=0$ for all $n<0$.
\end{proof}

\begin{rem}
The equivalence of (a) and (b) holds more generally for every finite \'etale cover of $D(\alpha)$, by  \cite[proof of Proposition~7.5]{deJongFundamental}.
\end{rem}

\begin{construction} \label{cons:non-pp-cov}
Fix two sequences of positive integers $(a_n)_{n\in\mathbf{Z}}$ and $(b_n)_{n\in\mathbf{Z}}$ such that $a_n/b_n>1$ for all $n\in\mathbf{Z}$ and $\lim_{|n|\to\infty} a_n/b_n = 1$. For $n\in\mathbf{Z}$, we set $Y_n$ to be the restriction of the Artin--Schreier covering $Y_{a_n,b_n}$ to $U^+$. By Lemma~\ref{lem:AS-splitting}, every $Y_n$ splits completely over $C=S(1)$, while for every $m>0$ its restriction to $D(1+\frac{1}{m})\cap U^+$ and $S(1+\frac{1}{m})$ is connected for $|n|\gg 0$. We set $Y^+ = \coprod_{n\in\mathbf{Z}} Y^+_n$. 

The automorphism $x\mapsto x^{-1}$ of $X$ induces an isomorphism $i\colon U^-\isomto U^+$. We let $Y^- = \coprod_{n\in\mathbf{Z}} Y^-_n$ be the pullback of $Y^+\to U^+$ under the map $i$. The restriction of $Y^-_n$ to $S(1-\frac{1}{m})$ is thus connected for $|n|\gg 0$, while $Y^-_n$ splits completely over $C$ for all $n$.

Label the irreducible components of $Y^\pm_n\times_{U^\pm} C$ by 
\[
    Z^\pm_{np}, \quad Z^\pm_{np+1}, \quad \ldots, \quad Z^\pm_{np+p-1} 
\]
(in any order); every $Z^\pm_m$ maps isomorphically onto $C$. Identify 
\[
    Y^+_n\times_{U^+} C = \coprod_{m\in\mathbf{Z}} Z^+_m
    \quad \text{with} \quad
    Y^-_n\times_{U^-} C = \coprod_{m\in\mathbf{Z}} Z^-_m
\]
by identifying $Z^+_m$ with $Z^-_{m-1}$ for all $m\in \mathbf{Z}$ (see Figure~\ref{fig:gluing}). This defines an \'etale morphism $Y\to X$ whose restriction to $U^+$ (resp.\ $U^-$) is $Y^+$ (resp.\ $Y^-$), in particular $Y\to X$ it is an object of $\Cov_X^\adm$.
\end{construction}

\begin{figure}
    \centering
    \includegraphics[{width=0.6\textwidth}]{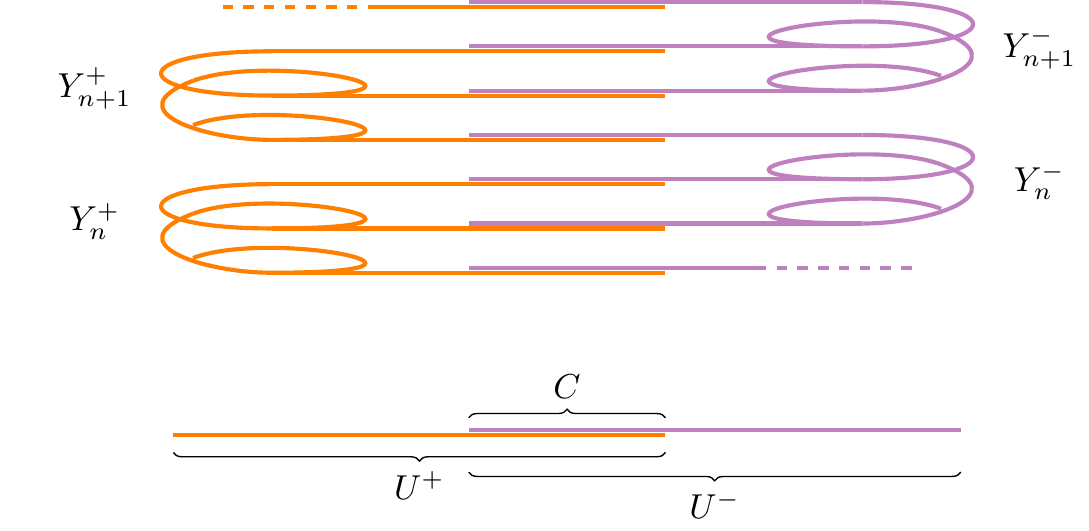}
    \caption{Construction of the covering $Y\to X$ (for $p=3$).}
    \label{fig:gluing}
\end{figure}

\begin{prop} \label{prop:counterex}
	The map $Y\to X$ does not belong to $\Cov^\oc_X$.
\end{prop}

\begin{proof}
Let $V$ be a connected overconvergent open subset of $X$, which by shrinking we may assume satisfies that $V\cap U^{\pm}$ is connected, containing the Gauss point $\eta$ of the unit disk. It suffices to show that $Y_V$ contains a connected open subset with infinite fiber size over $V$.

By Lemma~\ref{lem:open-in-Berkovich-disc} below, $V$ contains $S(1 + \frac{1}{m})\cup S(1-\frac{1}{m})$ for some $m>0$. Since for $n\gg 0$, say $n>n_0$, the restriction of $Y_n^\pm$ to $S(1 \pm \frac{1}{m})$ is connected, the restriction $Y_n^\pm \cap Y_V$ of $Y_n^\pm$ to $V\cap U^\pm$ is connected as well.

Let $Y'$ be the image in $Y$ of the union of $Y_n^+$ and $Y_n^-$ for all $n>n_0$, which is an open subset of $Y$. We claim that $Y'_V = Y'\cap Y_V$ is connected. This will give the required assertion since $Y'_V\to V$ has infinite fiber size.  To see the claim, it suffices to note that in the infinite sequence 
\[ 
    Y^+_{n_0+1}\cap Y_V, \quad Y^-_{n_0+1}\cap Y_V, \quad Y^+_{n_0+2}\cap Y_V, \quad Y^-_{n_0+2}\cap Y_V, \quad Y^+_{n_0+3}\cap Y_V, \quad \ldots
\]
each set is a connected open subset of $Y'\cap Y_V$ with a non-empty intersection with the subsequent one (e.g.\@ by considering the fiber over a point of $C$). Consulting Figure~\ref{fig:gluing} again might help the reader with the last step. 
\end{proof}

\begin{lem} \label{lem:open-in-Berkovich-disc}
    Every overconvergent neighborhood of $\eta$ inside $X$ contains $S(1 + \frac{1}{m})\cup S(1-\frac{1}{m})$ for some $m>0$.
\end{lem}

\begin{proof} 
We can replace $X$ with the ambient disc $D=D(|\varpi|^{-1})$. As the topological space $[D]$ agrees with the Berkovich disk $D^\mathrm{Berk}$, it suffices to prove the analogous claim for $D^{\rm Berk}$. By \cite[Proposition~1.6]{Baker}, a basis of the topology for $D^{\rm Berk}$ is given by sets of the form $D(a, r)^- \setminus \bigcup_{i=1}^n D(a_i, r_i)$  where $D(a, r)$ (resp.\ $D(a, r)^-$) denotes the open (resp.\ closed) disc with center $a$ and radius $r$, and where we allow $r>0$ or $n=0$. We note that for every $a, r$ we have either $D(a, r) = D(0, r)$ (if $|a|\leq |r|$) or $D(a, r)\subseteq S(0, |a|)$ where $S(0, |a|) = \{|x|=|a|\}$ (if $|a|>|r|$), and similarly for $D(a, r)^-$.

Let $V$ be an open subset of $D^{\rm Berk}$ containing  $\eta$ which is of the above form. By the above observation, we may assume that $a=0$, $r>\rho$, $a_1=0$, $r_1 < \rho$, and $D(a_i, r_i) \subseteq S(0, |a_i|)$ for $i\geq 2$. Let $\rho'\in (\rho, r)$ be such that the interval $(\rho, \rho')$ does not contain any of the $r_i$, then the annulus $\{\rho<|x|<\rho'\}$ is contained in $V$. Taking preimages in $D$ we obtain the desired claim.
\end{proof}

\subsection{Extending finite \'etale components}
\label{theorem 7.1 subsec}

The map $Y\to X$ was obtained by gluing together finite \'etale coverings $Y_n$ of an admissible open neighborhood $U$ of the point $\eta$ with the property that the maximal overconvergent open neighborhood $V_n$ of $U$ over which $Y_n$ extends shrinks to $U$ as $n$ tends towards infinity. We now show that this is a general phenomemon.

\begin{prop} \label{Theorem 7.1} 
    Let $X$ be a taut rigid $K$-space and let $f\colon Y\to X$ be \'etale and partially proper. Let $U$ be a quasi-compact open subset of $X$ and let $W$ be an open subset of $Y_U$ which is finite \'etale over $U$. Then, there exists an overconvergent open subset $V$ of $X$ containing $U$ and an open subset $W'$ of $Y_V$ which is finite \'etale over $V$ and such that $W'\cap Y_U=W$.
\end{prop}

To prove Proposition \ref{Theorem 7.1} we first establish Lemma \ref{topological 7.1 lemma} which shows the desired result at the level of universally separated quotients. We then upgrade this to rigid spaces using the general result in Proposition \ref{quasi-compact and universally closed relationship}. (The reader more familiar with Berkovich spaces will recognize that our line of argument can be carried out more directly on the level of $[X]=X^\mathrm{Berk}$.)

\begin{lem}\label{topological 7.1 lemma}
    Let $Y \to X$ be a map of Hausdorff topological  spaces with $Y$ locally compact. Let $Z \subset X$ be a subspace and let $W \subset Y_Z$ be a clopen and compact subspace. Then there exist an open subspace $V$ of $X$ containing $Z$ and a subspace $W'$ of $Y_V$ such that:
    \begin{enumerate}
        \item $W' \cap Y_Z = W$;
        \item $W'$ is open in $Y$,
        \item $W'\to V$ is a proper map of topological spaces.
    \end{enumerate}
\end{lem}

\begin{proof}
Let ${\mf B}$ be a collection of compact subspaces of $Y$ containing a neighborhood basis of every point of $Y$. It suffices to find finitely many $B_1, \ldots, B_m$ in ${\mf B}$ and an open $V$ containing $Z$ such that if $B=\bigcup_i B_i$ and $g\colon B\to X$ is the induced map, then:
    \begin{enumerate}
        \item $B \cap Y_Z = W$;
        \item $g^{-1}(V) \subset \interior_{Y}(B)$.
        \item $g\colon g^{-1}(V) \to V$ is a proper map of topological spaces.
    \end{enumerate}
as one may then take $W'=g^{-1}(V)$.

To prove the existence of such basis elements observe that as $Y_Z\setminus W$ is closed we may find a closed subspace $C$ of $Y$ such that $C\cap Y_Z=Y_Z\setminus W$. As $Y\setminus C$ is an open containing the compact set $W$, we may find finitely many $B_1,\ldots,B_m$ in $\mc{B}$ contained in $Y\setminus C$ such that $W\subseteq \interior_Y(B)$ where $B=\bigcup_i B_i$. Moreover, we see that $B\cap Y_Z=W$. Note that $g^{-1}(Z)=W$. Since $B$ is compact, $g$ is proper and so we may produce an open neighborhood $V$ of $Z$ such that $g^{-1}(V)\subseteq \interior_Y(B)$. Clearly $V$ satisfies the first two desired conditions, and the last property follows as $g$ is proper.
\end{proof}

\begin{prop}\label{quasi-compact and universally closed relationship}
    Let $f\colon Y\to X$ be a taut and valuative morphism of taut valuative spaces. Then $[f]$ is universally closed if and only if $f$ is quasi-compact.
\end{prop}
 \begin{proof} If $f$ is quasi-compact then $[f]$ is universally closed by \cite[Proposition 2.3.27, Chapter 0]{FujiwaraKato}. Suppose that $[f]$ is universally closed and let $U$ be a quasi-compact open subset of $X$. Note that as $X$ is quasi-separated $U$ is retrocompact and taut. By \cite[Proposition 2.2.5]{ALY1P1} we know that $[f^{-1}(U)]$ is homeomorphic to $[f]^{-1}(\sep_X(U))$. But, since $[f]$ is universally closed and $\sep_X(U)$ is quasi-compact, we know that $[f]^{-1}(\sep_X(U))$ is quasi-compact by \stacks{005R}. So, in particular, $[f^{-1}(U)]$ is quasi-compact. But, since $f^{-1}(U)$ is taut we know that the map 
\begin{equation*}
    \sep_{f^{-1}(U)}:f^{-1}(U)\to [f^{-1}(U)]
\end{equation*} is universally closed by \cite[Chapter 0, Theorem 2.5.7]{FujiwaraKato}, and so we deduce that $f^{-1}(U)$ is quasi-compact again by \stacks{005R}.
\end{proof}

\begin{proof}[Proof of Proposition \ref{Theorem 7.1}] By our assumptions, together with \cite[Theorem 2.5.7, Chapter 0]{FujiwaraKato} and \cite[Proposition 2.2.3 and Proposition 2.2.5]{ALY1P1}
\begin{itemize}[leftmargin=.6cm]
    \item $[f]\colon [Y]\to[X]$ is a map of locally compact Hausdorff spaces,
    \item $\sep_X(U)$ is a subspace of $[X]$ homeomorphic to $[U]$ (and similarly for $W$),
    \item $[Y_U]\to [f]^{-1}([U])$ is a homeomorphism,
    \item and $\sep_X(W)$ a compact clopen subset of $\sep_X(Y_U)$.
\end{itemize}
Thus, by Lemma \ref{topological 7.1 lemma}, one may find an open subspace $V_0$ of $[X]$ containing $[U]$ and a subspace $W'_0$ of $[Y]_{V_0}$ such that $W'_0$ is open in $[Y]$, $W'_0\to V_0$ is proper, and $W_0'\cap [f]^{-1}([U])=[W]$. Set 
\begin{equation*}
    V=\sep_X^{-1}(V_0)\subseteq X,\qquad W'=\sep_Y^{-1}(W_0')\subseteq Y,
\end{equation*} which are overconvergent open subsets of their ambient spaces. Note then that by construction $[W']\to [V]$ is proper. By Proposition \ref{quasi-compact and universally closed relationship} and the fact that $Y\to X$ is \'etale and partially proper, the morphism $W'\to V$ is \'etale, partially proper, and quasi-compact and so finite \'etale by \cite[Proposition 1.5.5]{Huberbook}. We finally claim that $W'\cap Y_U=W$. But, since $W'\cap Y_U$ and $W$ are overconvergent open subsets of $Y_U$, this can be checked in $[Y]$ from where it's true by construction.
\end{proof}

\section{Variants of the de~Jong fundamental group}
\label{s:cov-tau}

In this section we examine the categories $\Cov^\tau_X$ extending de Jong's category $\Cov^\oc_X$, and show they have good geometric properties.

\subsection{\texorpdfstring{The categories $\Cov_X$ and $\Cov^\tau_X$}{The categories CovX and Cov tau X}}
\label{ss:covtau}

Throughout this subsection let us fix a rigid $K$-space $X$. We will consider below the following three Grothendieck topologies $\tau = \oc,\adm,\et$ on $\Et_X$, whose covers are:
\begin{itemize}
    \item ($\tau=\oc$, for `overconvergent') open covers by overconvergent opens,
    \item ($\tau=\adm$, for `admissible')  open covers,
    \item ($\tau=\et$) jointly surjective \'etale morphisms.
\end{itemize}
which are in increasing order of fineness. (One could also define the `overconvergent \'etale topology' whose covers are jointly surjective partially proper \'etale morphisms, but it defines the same topos as the overconvergent topology by \cite[Lemma~2.2.8]{Huberbook}. See also \cite[Lemma~2.3]{deJongFundamental}).

For an adic space $X$, let us denote by $\UFEt_X$ the category of all disjoint unions of finite \'etale $X$-spaces.

\begin{definition} \label{cov tau def}
    Define $\Cov^\tau_X$ to be the category of morphisms $Y\to X$ for which there exists a cover $U\to X$ in the $\tau$-topology so that $Y_U\to U$ is an object of $\UFEt_U$. Define $\UCov^\tau_X$ to be the category of all disjoint unions of objects of $\Cov^\tau_X$.
\end{definition}

In \cite[Definition 5.2.2]{ALY1P1} we introduced the notion of a \emph{geometric covering} of $X$. It may be defined as an \'etale and partially proper morphism $f\colon Y\to X$ satisfying the following `valuative criterion': for all smooth and separated rigid $L$-curves $C$, where $L$ is a non-archimedean extension of $K$, and all morphisms $C\to X_L$, any embedding $i\colon [0,1]\to [C]$ and lift of $[0,1)\to [Y_C]$ along $[f_C]$ can uniquely be extended to a lift of $i$. We denote by $\Cov_X$ the category of geometric coverings of $X$.

By \cite[Propositions 5.2.4, 5.2.6, and 5.2.7]{ALY1P1} being a geometric covering is \'etale local on the target, is closed under disjoint unions, and contains finite \'etale coverings. Therefore, the following proposition follows.

\begin{prop} 
    The containment $\UCov^\tau_X\subseteq \Cov_X$ holds.
\end{prop}

\subsection{A brief recollection of tame infinite Galois categories}\label{igc sec}

In the this section we briefly recall the setup of the theory of tame infinite Galois categories in the sense of  \cite[\S7]{BhattScholze}. 

\begin{definition} 
    Let $\mc{C}$ be a category and $F\colon \mc{C}\to\cat{Set}$ be a functor (called the \emph{fiber functor}). We then call the pair $(\mc{C},F)$ an \emph{infinite Galois category} if the following properties hold:
    \begin{enumerate}
        \item the category $\mc{C}$ is cocomplete and finitely complete,
        \item Each object $X$ of $\mc{C}$ is a coproduct of categorically connected objects of $\mc{C}$.
        \item There exists a set $S$ of connected objects of $\mc{C}$ which generates $\mc{C}$ under colimits.
        \item The functor $F$ is faithful, conservative, cocontinuous, and finitely continuous.
    \end{enumerate}
    The \emph{fundamental group of $(\mc{C},F)$}, denoted $\pi_1(\mc{C},F)$, is the group $\Aut(F)$ endowed with the compact-open topology. We say that $(\mc{C},F)$ is \emph{tame} if for every categorically connected object $X$ of $\mc{C}$ the action of $\pi_1(\mc{C},F)$ on $F(X)$ is transitive.
\end{definition}

In the above we used the terminology that an object $Y$ of a category $\mc{C}$ is \emph{categorically connected}. This, by definition, means that if $Y'$ is a non-initial object of $\mc{C}$ then every monomorphism $Y'\to Y$ is an isomorphism. 

The upshot of the theory of (tame) infinite Galois categories is as follows.

\begin{prop}[{\cite[Theorem 7.2.5.(d)]{BhattScholze}}] \label{igc prop} 
    Let $(\mc{C},F)$ be a tame infinite Galois category. Then, the functor
    \begin{equation*}
        F\colon \mc{C}\isomto \GSet{\pi_1(\mc{C},F)}
    \end{equation*}
    is an equivalence. 
\end{prop}

From Proposition \ref{igc prop} we see that for a tame infinite Galois category $(\mc{C},F)$ the isomorphism classes of connected objects form a set. It is moreover clear that every object $Y$ isomorphic to a~disjoint union of its \emph{categorically connected components} which, by definition, are the connected subobjects, excluding the initial object. 

Using this, it is easy to deduce the following criterion for when a subcategory of a tame infinite Galois category is itself tame infinite Galois.

\begin{lem}\label{subcategory is tame}
    Let $(\mc{C},F)$ be a tame infinite Galois category. Let $\mc{C}'$ be a strictly full subcategory of $\mc{C}$ satisfying the following three properties:
    \begin{enumerate}[(a)]
        \item an object $X$ of $\mc{C}'$ is categorically connected as an object of $\mc{C}$ if and only if it is categorically connected as an object of $\mc{C}'$,
        \item an object $X$ of $\mc{C}$ is isomorphic to an object of $\mc{C}'$ if and only if its categorically connected components are isomorphic to objects of $\mc{C'}$,
        \item the subcategory $\mc{C}'$ is closed under (small) colimits and finite limits. 
    \end{enumerate}
    Then, $(\mc{C}',F)$ is a tame infinite Galois category.
\end{lem}

\subsection{\texorpdfstring{The categories $\UCov_X^\tau$ are tame infinite Galois categories}{The categories UCov tau X are tame infinite Galois categories}} 
\label{ss:covtau-galois}

We now demonstrate that $\UCov^\tau_X$ is a tame infinite Galois category.

\begin{thm} \label{taus are tame} 
    Let $X$ be a connected rigid $K$-space and $\ov{x}$ a geometric point in $X$. Then, $(\UCov^\tau_X,F_{\ov{x}})$ is a tame infinite Galois category.
\end{thm}

\begin{proof} 
It suffices to verify conditions (a)-(c) of Lemma \ref{subcategory is tame} with $\mc{C}=\Cov_X$ and $\mc{C}'=\UCov_X^\tau$. As (b) is clear, we focus only on (a) and (c).

By \cite[Proposition 5.4.5]{ALY1P1} it suffices to show that an object of $\UCov^\tau_X$ is categorically connected if and only if it is connected. But, by the method of proof of loc.\@ cit.\@ it then suffices to prove that for $Y_i\to X$ for $i=1,2$ and $Z\to X$ objects of $\UCov_X^\tau$ that the fiber product $Y_1\times_Z Y_2$ as adic spaces is an object of $\UCov_X^\tau$. One quickly reduces to the case when all three spaces lie in $\UFEt_X$, from where the claim is clear (cf.\@ \stacks{0BN9}).

To verify condition (c) it suffices to consider the case of either a coproduct, coequalizer, or fibered product diagram. The claim in the case of fibered products and disjoint unions follows from our discussion of condition (a). Thus, it suffices to prove this result in the case of a coequalizer diagram. But, the formation of coequalizers commutes with base change (cf.\@ \stacks{03I4}). In particular, by considering the base change to an appropriate $\tau$-cover we reduce ourselves to showing the coequalizer of a~diagram in $\UFEt_X$ taken in $\Cov_X$ is in $\UFEt_X$. But, this is clear (cf.\@ \stacks{0BN9}).
\end{proof}

\begin{definition}
    Let $\ov{x}$ be a geometric point of a connected rigid $K$-space $X$. We define the \emph{$\tau$-adapted de~Jong fundamental group} of the pair $(X,\ov{x})$, denoted $\pi_1^{\dJ,\tau}(X,\ov{x})$, to be the fundamental group of the tame infinite Galois category $(\UCov^\tau_X,F_{\ov{x}})$.
\end{definition}

\begin{cor} \label{fiber functor embedding for covtau}
    Let $X$ be a connected rigid $K$-space and $\ov{x}$ a geometric point of $X$. Then, the functor 
    \begin{equation*}
        F_{\ov{x}}\colon \UCov^\tau_X\to \pi_1^{\dJ,\tau}(X,\ov{x})\text{-}\cat{Set}
    \end{equation*}
    is an equivalence of categories.
\end{cor}

\section{\texorpdfstring{The categories $\Cov^\et_X$ and $\Loc(X_\proet)$}{The category Cov et X and locally constant sheaves in the pro-\'etale topology}}
\label{s:cov-tau-loc}

Throughout this section we fix an analytic locally strongly Noetherian adic space $X$. We show that there is a natural equivalence between $\Cov^\et_X$ (whose definition is the same as in Definition \ref{cov tau def}) and the category $\cat{Loc}(X_\proet)$ of locally constant sheaves (see \cite[Tome III, Expos\'{e} IX, \S2.0]{SGA4}) in the pro-\'etale topology as in \cite{Scholzepadic}.

\subsection{The pro-\'etale site of a rigid space and classical sheaves}

We begin by recalling the pro-\'etale site $X_\proet$ of an adic space $X$. Moreover, we single out a class of sheaves in $\Sh(X_\proet)$, the analogue of classical sheaves from \cite{BhattScholze}, which play an important role in the proof of our main result.

\medskip

\paragraph{The pro-\'etale site} 
For a category $\mc{C}$, denote by $\cat{Pro}(\mc{C})$ the pro-completion of $\mc{C}$ as in \cite[\S6]{KashiwaraSchapira}. We write an object of $\cat{Pro}(\mc{C})$ as $\{U_i\}$. Identify $\mc{C}$ as a full subcategory of $\cat{Pro}(\mc{C})$ by sending $U$ to the constant system $\{U\}$. To distinguish constant objects from general objects we write the latter as $\msf{U}$.

Denote by $\ProEt_X$ the full subcategory of $\cat{Pro}(\Et_X)$ consisting of those objects $\msf{U}$ such that $\msf{U}\to X$ is pro-\'etale in the sense of \cite[Definition 3.9]{Scholzepadic}. Every object of $\ProEt_X$ has a presentation of the form $\{U_i\}_{i\in\mc{I}}$ where
\begin{itemize} \item $\mc{I}$ has a final object $0$ such that $U_0\to X$ is \'etale, 
\item the maps $U_j\to U_i$ for $i\geqslant j\geqslant 0$ are finite \'etale and surjective.
\end{itemize}
When speaking of a presentation of an object of $\ProEt_X$ we shall assume it is of this form. By \cite[Lemma 3.10 vii)]{Scholzepadic} the category $\ProEt_X$ admits all finite limits, a fact we use without further comment.

As in \cite[\S3]{Scholzepadic}, for an object $\msf{U}=\{U_i\}$ we denote by $|\msf{U}|$ the topological space $\varprojlim |U_i|$ and call it the \emph{underlying topological space} of $\msf{U}$. We call an object $\msf{U}$ of $\ProEt_X$ \emph{quasi-compact and quasi-separated} if its underlying space $|\msf{U}|$ is. Denote the subcategory of quasi-compact and quasi-separated objects of $\ProEt_X$ by $\ProEt_X^\mathrm{qcqs}$. By \cite[Proposition~3.12 i), ii), and v)]{Scholzepadic}, $\ProEt^\mathrm{qcqs}_X$ is closed under fiber products and every object of $\ProEt_X$ has a cover in $X_\proet$ by objects of $\ProEt^\mathrm{qcqs}_X$ which moreover can be assumed to be open embeddings on the underlying topological spaces. Denote by $X^\mathrm{qcqs}_\proet$ the induced site structure on $\ProEt^\mathrm{qcqs}_X$ from $X_\proet$. By \cite[Tome I, Expos\'{e} III, Th\`{e}orem\'{e} 4.1]{SGA4} the natural morphism of topoi $\Sh(X_\proet)\to\Sh(X_\proet^\mathrm{qcqs})$ is an equivalence.

Define the \emph{pro-\'etale site} of $X$, denoted $X_\proet$, to be the site whose underlying category is $\ProEt_X$ and whose topology is as in \cite[\S5.1]{BMS}. The inclusion functor $\Et_X\to \ProEt_X$ preserves finite limits, and is a continuous morphism of categories with Grothendieck topologies. So, we obtain an induced morphism of sites $X_\proet\to X_\et$. We denote by $\nu_X$, or $\nu$ when $X$ is clear from context, the induced morphism of topoi $\Sh(X_\proet)\to\Sh(X_\et)$.

\medskip

\paragraph{Classical sheaves} 

Using the morphism of topoi $\nu_X$, we can define the appropriate notion of `classical sheaf' as in \cite[Definition 5.1.3]{BhattScholze}. 

\begin{definition}
     A sheaf $\mc{G}$ in $\Sh(X_\proet)$ is called \emph{classical} if it is in the essential image of the pullback functor $\nu^\ast\colon \Sh(X_\et)\to\Sh(X_\proet)$.
\end{definition}

Denote by $\mc{F}^{\smallto}$ the association
\begin{equation*}
    \mc{F}^{\smallto}(\{U_i\})=\varinjlim \mc{F}(U_i)
\end{equation*}
which is an element of $\cat{PSh}(X_\proet)$. In fact, $\mc{F}^{\smallto}$ is the presheaf pullback of $\mc{F}$, and thus there is a natural map $\mc{F}^{\smallto}\to \nu_X^\ast(\mc{F})$ of objects of $\cat{PSh}(X_\proet)$.

\begin{prop}[{cf.\@ \cite[Lemma 3.16]{Scholzepadic}}] \label{classical sheaf computation} 
    For any sheaf $\mc{F}$ in $\Sh(X_\et)$ the map $\mc{F}^{\smallto}\to \nu_X^\ast(\mc{F})$ is a bijection when evaluated on any object of $X_\proet^\mathrm{qcqs}$.
\end{prop}

We thus obtain a structure theorem for classical sheaves.

\begin{prop}[{cf.\@ \cite[Lemma 5.1.2]{BhattScholze}}] \label{classical corollary} 
   For any sheaf $\mc{F}$ in $\Sh(X_\et)$, the unit map $\mc{F}\to \nu_\ast\nu^\ast\mc{F}$ for the adjunction $\nu^\ast\dashv \nu_\ast$ is an isomorphism. Therefore, the functor $\nu^\ast\colon \Sh(X_\et)\to\Sh(X_\proet)$ is fully faithful with essential image those $\mc{F}$ such that the counit map $\nu^\ast\nu_\ast\mc{F}\to \mc{F}$ is an isomorphism.
\end{prop}

\begin{proof} 
The latter statements follow from general category theory (see \cite[\S IV.1, Theorem~1]{MacLane} and \stacks{07RB}). To see the first statement we note that by \cite[Tome I, Expos\'{e}~III, Th\`{e}orem\'{e} 4.1]{SGA4} it suffices to show that the map $\mc{F}\to \nu_\ast\nu^\ast\mc{F}$ is an isomorphism when evaluated on any quasi-compact and quasi-separated object of $X_\et$. This follows from Proposition \ref{classical sheaf computation}.
\end{proof}

Arguing as in \cite[Lemma~5.1.4]{BhattScholze}, we obtain the following corollary.

\begin{lem} \label{lem:locally classical}
    Let $\mc{F}$ be an object of $\Sh(X_\proet)$. If there exists a covering $\{\msf{U}_i\to X\}$, and for each $i$ a classical sheaf $\mc{F}_i$ such that $\mc{F}|_{\msf{U}_i}\simeq \mc{F}_i|_{\msf{U}_i}$, then $\mc{F}$ is classical. In particular, locally constant sheaves are classical.
\end{lem}

From this and Lemma \ref{inverse limit connected component lemma} below, we obtain the following more concrete description of constant sheaves. This is not directly needed below.

\begin{prop} \label{constant sheaf computation} 
   For any set $S$ the natural map of sheaves 
   \begin{equation*}
       \underline{S}_{X_\proet}\to  \Hom_\cnts(\pi_0(|-|),S)
   \end{equation*} is a bijection 
   when applied to any object of $X_\proet^\mathrm{qcqs}$.
\end{prop}

\begin{lem} \label{inverse limit connected component lemma} 
    Let $\{X_i\}$ be a projective system of spectral spaces with quasi-compact transition maps. Set $X=\varprojlim_i X_i$. Then, $X$ is spectral and the natural map $\pi_0(X)\to \varprojlim_i \pi_0(X_i)$ is a homeomorphism of profinite spaces.
\end{lem}

\begin{proof}
Under the given assumptions, $X$ is spectral by \cite[Chapter 0, Theorem 2.2.10]{FujiwaraKato}. Suppose first that the $X_i$ are connected and non-empty. By \cite[Chapter 0, Proposition 3.1.10]{FujiwaraKato}, the induced map $\varinjlim \Gamma(X_i,\underline{\mathbf{F}_2})\to \Gamma(\varprojlim X_i,\underline{\mathbf{F}_2})$ is a bijection. Thus $\Gamma(X, \underline{\mathbf{F}_2}) = \mathbf{F}_2$, and so $X$ is connected and non-empty, showing the assertion.

By \stacks{0906}, $\pi_0(X)$ and each $\pi_0(X_i)$ are profinite, and thus also $\varprojlim_i \pi_0(X_i)$ is profinite. Since profinite spaces are compact, it suffices to prove that the map in question is a bijection. To this end, let $(C_i)$ be in $\varprojlim_i \pi_0(X_i)$. As $C_i$ is closed in $X_i$, the $C_i$ form an inverse system of non-empty connected spectral spaces with quasi-compact transition maps, and hence by the first paragraph the space $C = \varprojlim_i C_i$ is a non-empty and connected spectral space. 

By \cite[Chapter 0, Lemma 2.2.19]{FujiwaraKato} we have $C = \bigcap_i p_i^{-1}(C_i)$, where $p_i:X\to X_i$ is the natural map. Let $x$ be a point in $C$ and let $C_x$ be the component of $X$ containing it. Since $p_i(C_x)$ is connected and intersects $C_i$ at $p_i(x)$ we see that $p_i(C_x)\subseteq C_i$ for all $i$, so that $C_x$ maps to $(C_i)$. Conversely, if the connected components of $C_x$ and $C_y$ for $x,y\in X$ both map to $(C_i)$, then $x,y\in \bigcap_i p_i^{-1}(C_i)=C$, which is connected, and so $C_x=C_y$.
\end{proof}

For an object $Y$ of $\Et_X$ we denote by $h_{Y,\et}$ its corresponding sheaf. For an object $\msf{Y}$ of $\ProEt_X$ we denote by $h_{\msf{Y},\proet}$ its corresponding presheaf, and by $h_{\msf{Y},\proet}^\sharp$ its sheafification. This sheafification is necessary.

\begin{example}
Let $X$ be the disjoint union of copies $X_n$ of $\Spa(K)$ indexed by $n\geq 0$, and let $L$ be a non-trivial finite separable extension of $K$. We set $Y$ to be the disjoint union of $Y_n=\Spa(L)$, with the natural \'etale map $Y\to X$. Let $\msf{U} = \{U_i\}$ be the following object of $X_\proet$ indexed by $i\geq 0$: the space $U_i$ is the disjoint union of $U_{i,n}$ indexed by $n\geq 0$, where $U_{i,n}= X_n$ for $n>i$ and $U_{i,n}= Y_n$ for $n\leq i$. Note that for every $i$, we have $\Hom_X(U_i, Y)=\emptyset$, and hence $h_{Y,\proet}(\msf{U})$ is empty. Then $\{Y_n\to \msf{U}\}_{n\geq 0}$ forms a pro-\'etale cover which violates the sheaf condition for $h_{Y,\proet}$.
\end{example}

While $X_\proet$ is not subcanonical, we can use Proposition \ref{classical sheaf computation} to show that $h_{\msf{Y},\proet}$ is a sheaf when restricted to $X^\mathrm{qcqs}_\proet$.

\begin{prop}\label{prop:almost-representable-presheaf}
    Let $\msf{Y}$ be an object of $\ProEt_X$. Then, the natural map  $h_{\msf{Y},\proet}\to h^\sharp_{\msf{Y},\proet}$ a bijection when evaluated on any element of $\ProEt^\mathrm{qcqs}_X$. In particular, the site $X^\mathrm{qcqs}_\proet$ is subcanonical.
\end{prop}

\begin{proof} By Lemma \ref{lem:site-basis} below, it suffices to show that $h_{\msf{Y},\proet}$ is a sheaf when restricted to $X^\mathrm{qcqs}_\proet$. But, if $\msf{Y}=\{Y_i\}$ then  $h_{\msf{Y},\proet}=\varprojlim h_{Y_i,\proet}$ where this inverse limit is taken in $\cat{PSh}(X_\proet)$. Since the inverse limit of sheaves, taken in the category of presheaves, is a sheaf, we're restricted to showing that for all $i$ the restriction of $h_{Y_i,\proet}$ to $X^\mathrm{qcqs}_\proet$ is a sheaf. But, since $h_{Y_i,\proet}=h_{Y_i,\et}^{\smallto}$ this follows from Proposition \ref{classical sheaf computation}.
\end{proof}

\begin{lem} \label{lem:site-basis}
    Let $\mc{C}$ be a site and let $\mc{B}\subseteq\mc{C}$ be a full subcategory closed under fiber products and such that every object $V$ of $\mc{C}$ admits a covering family $\{U_\alpha\to V\}_{\alpha\in I}$ with each $U_\alpha$ an object of $\mc{B}$. Let $\mc{F}$ be a presheaf on $\mc{C}$ whose restriction to $\mc{B}$ is a sheaf, and let $\mc{F}^\#$ be its sheafification. Then, for every object $U$ of $\mc{B}$, the natural map $\mc{F}(U)\to \mc{F}^\#(U)$ is an isomorphism.
\end{lem}

\begin{proof}
For an object $V$ of $\mc{C}$, denote by $\mc{J}_V$ the category whose objects are covering families $\{V_\alpha\to V\}_{\alpha\in I}$ and whose morphisms are refinements. We note that by \stacks{00W7} and \stacks{00W6} that the diagram $H^0(-,\mc{F})\colon \mc{J}_V\to \cat{Set}$ is filtered (see \stacks{002V}) and we define $\mc{F}^+(V) = \varinjlim_{\mc{J}_V} H^0(\mc{V}, \mc{F})$. By \stacks{00W1} the sheafification of $\mc{F}$ is identified with the composition $\mc{F}\to \mc{F}^+\to (\mc{F}^+)^+$.

Denote by $\mc{J}'_V$ the full subcategory of $\mc{J}_V$ consisting of coverings whose elements are in $\mc{B}$. The diagram $H^0(-,\mc{F})\colon \mc{J}'_V\to\cat{Set}$ is cofinal in the diagram $H^0(-,\mc{F})\colon \mc{J}_V\to\cat{Set}$. Therefore $\mc{F}^+(V) = \varinjlim_{\mc{J}'_V} H^0(\mc{V}, \mc{F})$. This colimit depends only on the values of $\mc{F}$ on $V$, and if $V$ is itself an object of $\mc{B}$, then $\mc{F}^+(V)=\mc{F}(V)$. Therefore $\mc{F}^\#(V) = (\mc{F}^+)^+(V) = \mc{F}^+(V) = \mc{F}(V)$.
\end{proof}

\subsection{\texorpdfstring{Sheaves on $\pfset{G}$ and $X_\profet$}{Sheaves on G-pfset and Xprofet}}

In this subsection we establish a profinite version of our main result.

\medskip 

\paragraph{The site $\pfset{G}$}

For a profinite group $G$, denote by $G\text{-}\cat{FSet}$ the category of finite sets with a continuous action of $G$. One has a natural identification $\cat{Pro}(G\text{-}\cat{FSet})$ with $\pfset{G}$ where the latter category is the category of profinite topological spaces with a continuous action of $G$.

We endow $\pfset{G}$ with a Grothendieck topology where the coverings $\{S_i\to S\}$ are jointly surjective maps such that each map $S_i\to S$ satisfies a condition similar to that in the definition of the pro-\'etale site as in \cite[\S5.1]{BMS} (with `$\msf{U}_0\to\msf{U}$ \'etale' replaced by `$\msf{U}_0 \to\msf{U}$ is the pullback of a morphism in $G\text{-}\cat{FSet}$' and `$\msf{U}_\mu\to \msf{U}_{<\mu}$ is finite \'etale surjective' is replaced by  `$\msf{U}_\mu\to\msf{U}_{<\mu}$ is the pullback of a finite surjective map in $G\text{-}\cat{FSet}$').

\begin{prop}[{\cite[(1)]{ScholzeErratumToRigid}}] \label{group covering lem} 
   The collection $\{G\to \ast\}$ is a cover.
\end{prop}

\begin{prop}\label{prop:G-pfset-finite-colim-subcanonical}
    The site $\pfset{G}$ is finite cocomplete and subcanonical.
\end{prop}

\begin{proof} 
Since one has an equivalence $\pfset{G}\simeq \cat{Pro}(G\text{-}\cat{FSet})$ and $G\text{-}\cat{FSet}$ has all finite colimits, the existence of all finite colimits in $\pfset{G}$ follows from \cite[Corollary 6.1.17 i)]{KashiwaraSchapira}. To show $\pfset{G}$ is subcanonical, let $S$ be an element of $\pfset{G}$. We can write $S=\varprojlim S_i$ where $S_i$ are finite objects of $\pfset{G}$. Since $h_S=\varprojlim_i h_{S_i}$ where the right hand side is a limit in the category of presheaves, we are reduced to showing that $h_S$ is a sheaf when $S$ is finite. This is handled in the proof of Proposition \ref{locally constant on profinite g-sets} below.
\end{proof}

\medskip
\paragraph{Locally constant sheaves on $\pfset{G}$}

We now describe the category $\cat{Loc}(\pfset{G})$ in terms of the category $G\text{-}\cat{Set}$.

\begin{prop} \label{locally constant on profinite g-sets}
    Let $G$ be a profinite topological group. Then, the functor
    \begin{displaymath}
    \GSet{G}\to \Loc\left(\pfset{G}\right),\qquad T\mapsto (\mc{F}_T:S\mapsto \Hom_{\cnts,G}(S,T))
    \end{displaymath}
    is an equivalence of categories.
\end{prop}

\begin{proof}
We first show that $\mc{F}_T$ is a sheaf. As covers in $\pfset{G}$ are jointly surjective, it follows that the sheaf condition is satisfied for the presheaf sending $S$ to $\Hom_G(S,T)$ on $\pfset{G}$. One needs to check that continuity condition is local, which can be checked as follows. Any cover $\{U_i \to U\}$ in $\pfset{G}$ has a finite subcover. Passing to such subcover and considering $\coprod_i U_i \to U$, we can assume the map to be a surjective map of compact spaces. It is then automatically a quotient map of topological spaces, thus the continuity of functions can be tested after pullback to $\coprod_i U_i$.

For $S$ (resp.\@ $T$) an object of $\pfset{G}$ (resp.\@ $\GSet{G}$), denote the underlying set of $S$ (resp.\@ $T$) with trivial $G$-action by $S_{\rmta}$ (resp.\@ $T_{\rmta}$). We claim that $\mc{F}_T$ is trivialized after restricting to $G$ seen as an element of $\pfset{G}$. Indeed, the slice category $\pfset{G}/G$ is equivalent to the category of profinite sets $\pfset{*}$ via $(S \stackrel{f}{\to} G) \mapsto (f^{-1}(1_G))$. Under this equivalence, the pushforward of $\mc{F}_T$ is the sheaf on $\pfset{*}$ given by $U \mapsto \mathrm{Maps}_{\cnts}(U,T)$, which is precisely a constant sheaf in $\pfset{*}$.

Suppose now that $\mc{F}$ is a locally constant sheaf on $\pfset{G}$. We now show there exists $T$ in $\GSet{G}$ such that $\calF = \calF_T$. We claim that the obvious map $\colim \mc{F}(S)\to \mc{F}(S_i)$, where $S=\varprojlim S_i$ where $S_i$ are finite $G$-sets, is a bijection, which we refer to as $\mc{F}$ being classical. To show such an equality it suffices to check it after a cover (cf.\@ \cite[Lemma 5.1.4]{BhattScholze}). This reduces us to the constant case, which is trivial.

We define our candidate for $T$ as follows. As a discrete set, we set $T = \calF(G)$. The $G$-action on $T = \calF(G)$ is defined using the map $G \to \Aut_{\pfset{G}}(G)^{\mathrm{op}}$ given by $h \mapsto (g \mapsto g\cdot h)$ (where we view $G$ as an element of $\pfset{G}$ by acting on the left). To see that this action is continuous we must show the equality $\mc{F}(G)=\colim \mc{F}(G/U)$, as $U$ travels over the open normal subgroups of $G$, but this follows from classicality.

We now verify that $\mc{F}_T$ is isomorphic to $\mc{F}$. By Lemma \ref{group covering lem} the map $G \to *$ is a cover, and thus $G \times S \to S$ is a cover for any $S$ an object of $\pfset{G}$. Observe that $G \times S_{\rmta} \to S$ defined by $(g,s) \mapsto gs$ is isomorphic to $G \times S \stackrel{\mathrm{pr_S}}{\to} S$ in $\pfset{G}/S$ via the map $(g,s) \mapsto (g,gs)$. Similarly, $G \times G_{\rmta} \times S_{\rmta}$ is isomorphic to $G \times G \times S$ via the map $(g,h,s) \mapsto (g, gh, gs)$. We thus have the following isomorphism of diagrams. \begin{equation}\label{diagram identification}
    \xymatrixcolsep{5pc}\xymatrix{
        S\ar@{=}[d]  & \ar[l]\ar[d]^{\rotatebox{90}{$\sim$}} G \times S_{\rmta} & \ar@<-.5ex>[l] \ar@<.5ex>[l] \ar[d]^{\rotatebox{90}{$\sim$}} G \times G_{\rmta} \times S_{\rmta}\\ 
        S  & \ar[l]G \times S & \ar@<-.5ex>[l] \ar@<.5ex>[l]  G \times G \times S.
    }
\end{equation}
By the classicality of $\mc{F}$ we have a canonical identification \begin{equation*}
    \calF(G \times S_{\rmta}) = \colim_i \calF(G \times S_{i}) = \mathrm{Maps}(S_i,T) = \Hom_{\cnts}(S_\rmta,T)
\end{equation*} 
after presenting $S_{\rm ta} = \lim_i S_i$ for finite (discrete) $S_{i}$, and the middle equality follows from the canonical identifications 
\begin{equation*}
    \mc{F}(G \times S_i) = \mc{F}(\coprod_{s \in S_i} G) = \prod_{s \in S_i}\mc{F}(G) = \mathrm{Maps}(S_i,T).
\end{equation*}
Consider the exact sequence of sets
\begin{displaymath}
\calF(S) \to \calF(G \times S_{\rmta}) \rightrightarrows \calF(G \times G_{\rmta} \times S_{\rmta})
\end{displaymath}
obtained by using the identification given in Equation \eqref{diagram identification}, the observation that as an object of $\GSet{G}$ we have that $(G\times S)\times_S (G\times S)$ is isomorphic to $G\times G\times S$, and the sheaf sequence for the cover $G\times S\to S$. We then make the identifications
\begin{equation}\label{identification sequence Loc proof}\calF(G \times S_{\rmta}) = \Hom_{\cnts}(S_{\rmta},T),\qquad \calF(G \times G_{\rmta} \times S_{\rmta}) = \Hom_{\cnts}(G_{\rmta} \times S_{\rmta}, T)
\end{equation}
as above. As the maps $G \times G_{\rmta} \times S_{\rmta} \rightrightarrows G \times S_{\rmta}$ are explicitly given by $(g,h,s) \mapsto (g,s)$ and $(g,h,s) \mapsto (gh,h^{-1}s)$, and by the definition of the action of $G$ on $T$, we see that the corresponding maps 
\begin{equation*}
    \Hom_{\cnts}(S_{\rmta},T) \rightrightarrows \Hom_{\cnts}(G_{\rmta} \times  S_{\rmta}, T)
\end{equation*} are given by $f \mapsto f \circ \mathrm{pr}_{S_{\rmta}}$ and $f \mapsto ((h,s) \mapsto h\cdot f(h^{-1} \cdot s))$. Using this, and the sequence given in Equation \eqref{identification sequence Loc proof} we get the canonical identification of $\calF(S)$ and $\Hom_{\cnts,G}(S,T) = \calF_T(S)$, as desired.

We have shown that our functor is essentially surjective, and showing that it is fully faithful is routine. 
\end{proof}

\medskip
\paragraph{The site $X_\profet$}

Consider $\cat{Pro}(\FEt_X)$ as a full subcategory of $\ProEt_X$ and endow it with a Grothendieck topology as in \cite[\S5.1]{BMS} (with `$\msf{U}_0\to\msf{U}$ \'etale' replaced by `$\msf{U}_0\to\msf{U}$ finite \'etale').

\begin{prop}[{\cite[Proposition 3.5]{Scholzepadic},\cite{ScholzeErratumToRigid}}] \label{profinite site same thing as profinite sets} 
    Let $X$ be a connected adic space and let $\ov{x}$ be a geometric point of $X$. The functor
    \begin{equation*}
        X_\profet\to \pfset{\pi_1^\alg(X,\ov{x})},\qquad \{U_i\}\mapsto \varprojlim F_{\ov{x}}(U_i)
    \end{equation*}
    is an equivalence of sites.
\end{prop}

From this and Proposition \ref{prop:G-pfset-finite-colim-subcanonical} we immediately deduce the following.

\begin{cor}\label{cor:profet-finite-colim-subcanonical}
    The site $X_\profet$ has all finite colimits and is subcanonical.
\end{cor}

\medskip 
\paragraph{Localy constant sheaves on $X_\profet$}

We end this section by describing all the objects of $\cat{Loc}(X_\profet)$ using Proposition \ref{locally constant on profinite g-sets}. Define for any object $\msf{Y}=\{Y_j\}$ of $\cat{Pro}(\Et_X)$ the presheaf $h_{\msf{Y}, \profet}$ on $X_\profet$ obtained by restricting $h_{\msf{Y},\proet}$ to $X_\profet$. More explicitly, we have the following formula
\begin{equation*}
    h_{\msf{Y},\profet}(\{U_i\})=\varprojlim_j \varinjlim_i \Hom_{\Et_X}(U_i,Y).
\end{equation*}
We define $h_{\msf{Y},\profet}^\sharp$ to be the sheafification of this presheaf. The following result then follows from Proposition \ref{profinite site same thing as profinite sets}, Proposition \ref{locally constant on profinite g-sets}, and the fact that $\pi_1^\alg(X,\ov{x})\text{-}\cat{Set}\simeq \UFEt_X$.

\begin{prop}\label{prop:locally-constant-profet}
    For all $Y$ in $\UFEt_X$, the presheaf $h_{Y,\profet}$ is a sheaf and the functor
    \begin{equation*}
        \UFEt_X\to \cat{Loc}(X_\profet),\qquad Y\mapsto h_{Y,\profet}
    \end{equation*}
    is an equivalence of categories.
\end{prop}
\begin{proof}
It suffices to verify the first claim, and to assume $X$ is connected. The presheaf $h_{Y,\profet}$ corresponds via the equivalence in Proposition \ref{profinite site same thing as profinite sets} to the sheaf $\mc{F}_T$ from Proposition~\ref{locally constant on profinite g-sets} with $T=Y_{\ov{x}}$, and thus is a~sheaf. 
\end{proof}

\subsection{\texorpdfstring{Sheaves on $X_\proet$ which are pro-finite \'etale locally constant}{Sheaves on X proet which are pro-finite \'etale locally constant}}

We now to upgrade Proposition \ref{prop:locally-constant-profet} to a result about sheaves on $X_\proet$ which are pro-finite \'etale locally constant.

\medskip
\paragraph{Comparing $\Sh(X_\proet)$ and $\Sh(X_\profet)$}

The inclusion $\cat{Pro}(\FEt_X)\to \ProEt_X$ is a continuous functor of categories with Grothendieck topologies. Since this morphism preserves fiber products we get an induced map of sites $X_\proet\to X_\profet$. Denote the induced morphism of topoi $\Sh(X_\proet)\to\Sh(X_\profet)$ by $\theta_X$ or just $\theta$ when $X$ is clear from context. 

\begin{prop}[{cf.\@ \cite[Proposition VI.9.18]{AbbesGros}}] \label{prop:unit-isomorphism} 
    Let $X$ be a quasi-compact and quasi-separated adic space. Then, the unit map $\mathbf{1}\to \theta_\ast\theta^\ast$ for the adjunction $\theta^\ast\dashv \theta_\ast$ is an isomorphism. Thus, $\theta^\ast\colon \Sh(X_\profet)\to \Sh(X_\proet)$ is fully faithful with essential image those sheaves $\mc{F}$ for which the counit map $\theta^\ast\theta_\ast\mc{F}\to\mc{F}$ is an isomorphism.
\end{prop}

\begin{proof} 
Again, the latter statements follow from general category theory. To show the first statement begin by noting that $\theta^\ast$, being a left adjoint, commutes with colimits. We claim that $\theta_\ast$ also commutes with filtered colimits. By \cite[Tome II, Expos\'{e} VI, Th\`{e}orem\'{e} 5.1]{SGA4} it suffices to prove that $\theta$ is a coherent morphism between coherent topoi. The fact that $\Sh(X_\proet)$ is coherent is verified in \cite[Proposition 3.12 (vii)]{Scholzepadic}, and the fact that $\Sh(X_\profet)$ is coherent follows easily from this. To show that $\theta$ is coherent let $\msf{V}=\{V_j\}$ be an object of $X_\profet$. This system consists of spectral spaces and the transition maps are quasi-compact. We deduce from Proposition \ref{inverse limit connected component lemma} that $|\msf{V}|$ is quasi-compact and quasi-separated from where we are finished by \cite[Proposition 3.12 (iv)]{Scholzepadic}.

Fix a sheaf $\mc{F}$ in $\Sh(X_\proet)$. We now show that the unit map $\mc{F}\to\theta^\ast\theta_\ast\mc{F}$ is an isomorphism. By Corollary \ref{cor:profet-finite-colim-subcanonical} the category $X_\profet/\mc{F}$ is filtered, where this category means the subcategory of the slice category $\cat{PSh}(X_\profet)/\mc{F}$ consisting of representable presheaves (see \cite[Tome I, Expos\'{e} I, 3.4.0]{SGA4}). By \cite[Tome I, Expos\'{e} II, Corollaire 4.1.1]{SGA4} the natural map 
\begin{equation*}
    \varinjlim_{\msf{Y}\in\cat{Pro}(\FEt_X)/\mc{F}} h_{\msf{Y},\profet}\to \mc{F}
\end{equation*} 
is an isomorphism where we do not need to sheafify $h_{\msf{Y},\profet}$ by Corollary \ref{cor:profet-finite-colim-subcanonical}. Since $\theta^\ast$ and $\theta_\ast$ both commute with filtered colimits we've thus reduced to showing that the unit map is an isomorphism when evaluated on representable sheaves $h_{\msf{Y},\profet}$. A computation using the adjunction property and \stacks{04D3} shows this is equivalent to the bijectivity of the map $h_{\msf{Y},\profet}(\msf{V})\to h_{\msf{Y},\proet}^\sharp(\msf{V})$. This follows from Proposition \ref{prop:almost-representable-presheaf}.
\end{proof}

\begin{prop} \label{prop:pullback-ufet-object}
    Let $X$ be a quasi-compact and quasi-separated adic space and let $Y$ be an object of $\UFEt_X$. Then, $\theta^\ast h_{Y,\profet}\simeq h_{Y,\proet}^\sharp$.
\end{prop}

\begin{proof} 
Write $Y=\varinjlim_j Y_j$ with $Y_j$ objects of $\FEt_X$ and the transition maps are open embeddings. Since $\theta^\ast$ commutes with colimits, we're reduced to showing that 
\begin{equation*}
    h_{Y,\profet} \simeq \varinjlim_j h_{Y_j,\profet}\quad \text{and }\quad h_{Y,\proet}^\sharp\simeq \varinjlim_j h_{Y_j,\proet}^\sharp.
\end{equation*} 
The former is clear since every object of $X_\profet$ is quasi-compact and quasi-separated, and the latter follows from combining Proposition \ref{prop:almost-representable-presheaf} with \cite[Tome I, Expos\'{e} III, Th\`{e}orem\'{e} 4.1]{SGA4}.
\end{proof}

\medskip 
\paragraph{Locally constant sheaves on $X_\proet$ trivialized on a pro-finite \'etale cover} 

We now establish the special case of our main theorem, classifying locally constant sheaves on $X_\proet$ which become constant on a pro-finite \'etale cover of $X$.

\begin{prop}\label{prop:main-theorem-locally-constant-profinite-case}
    Let $X$ be a quasi-compact and quasi-separated adic space. Then, the functor
    \begin{equation*}
        \UFEt_X\to \cat{Loc}(X_\proet),\qquad Y\mapsto h^\sharp_{Y,\proet}
    \end{equation*}
    is fully faithful with essential image those objects $\mc{F}$ of $\cat{Loc}(X_\proet)$ which become constant on a~pro-finite \'etale cover of $X$.
\end{prop}

\begin{proof} 
We may assume that $X$ is connected. To see this functor has image contained in the correct subcategory of $\cat{Loc}(X_\proet)$ let $Y$ be an object of $\UFEt_X$. As observed in the proof of Proposition \ref{prop:locally-constant-profet} the sheaf $h_{Y,\profet}$ corresponds, under Proposition \ref{profinite site same thing as profinite sets} to $\mc{F}_T$ where $T=Y_{\ov{x}}$. From loc.\@ cit.\@ we see that $h_{Y,\profet}$ is an object of $\cat{Loc}(X_\profet)$. In particular, $\theta^\ast h_{Y,\profet}$ is an object of $\cat{Loc}(X_\proet)$ trivialized on a pro-finite \'etale cover of $X$ and so we're done by Proposition \ref{prop:pullback-ufet-object}. Moreover, our functor is fully faithful since it can be described, by Proposition \ref{prop:almost-representable-presheaf}, as $\nu^\ast$ restricted to $\UFEt_X$ from where we deduce fully faithfulness by Proposition \ref{classical corollary}.

To show our functor is essentially surjective let $\mc{F}$ an object of $\Sh(X_\proet)$ trivialized on a pro-finite \'etale cover of $X$. In particular, it is trivialized on a fixed universal pro-finite Galois cover $\tilde{X}$ of $X$ which, under the equivalence in Proposition \ref{profinite site same thing as profinite sets}, corresponds to $\pi_1^\alg(X,\ov{x})$ acting on itself by left multiplication. By combining Proposition \ref{profinite site same thing as profinite sets} and Proposition \ref{locally constant on profinite g-sets}, there exists some $Y$ in $\UFEt_X$ and an isomorphism $ \psi : h_{Y,\profet} \isomto \theta_{*} \mc{F}$ of objects of $\Sh(X_\profet)$. One then obtains a morphism 
\begin{equation*}
    h^\sharp_{Y,\proet} \isomto \theta^*h_{Y,\profet} \isomto  \theta^*\theta_*\mc{F} \to \mc{F}
\end{equation*}
where the first isomorphism comes from Proposition \ref{prop:pullback-ufet-object}, the second map is $\theta^\ast \psi$, and the last map is the counit map. We claim that this is an isomorphism. It remains to check that the counit map $\phi : \theta^*\theta_*{\mc F} \to {\mc F}$ is an isomorphism. To see this, first observe that since this is a morphism of sheaves and $\tilde{X} \to X$ is a cover in $X_\proet$, it is enough to check that $\phi_{|\tilde{X}}$ is an isomorphism. As both sheaves are locally constant and trivialized when restricted to the connected cover $\tilde{X}$, it is enough to check that $\phi(\tilde{X})\colon\theta^*\theta_*{\mc F}(\tilde{X}) \to \mc{F}(\tilde{X})$ is a bijection. This can be checked after applying $\theta_*$. The map obtained from unit and counit $\theta_* \to \theta_* \theta^*\theta_* \to \theta_*$ is the identity morphism (see \cite[\S IV.1, Theorem 1]{MacLane}). By Proposition \ref{prop:unit-isomorphism}, we know that the first morphism in this composition is an isomorphism. It follows that the second morphism is an isomorphism too. Thus, $\theta_* \theta^*\theta_*{\mc F}(\tilde{X}) = \theta_*{\mc F}(\tilde{X})$, as desired.
\end{proof}

\subsection{Main result}

We now arrive at the main result of this section. 

\begin{thm} \label{main locally constant theorem} 
    Let $X$ be an adic space. Then, the functor
    \begin{equation*}
        \Cov_X^\et\to \cat{Loc}(X_\proet),\quad Y\mapsto h^\sharp_{Y,\proet}
    \end{equation*}
    is an equivalence of categories. 
\end{thm}

\begin{proof} The final claim follows immediately Corollary \ref{fiber functor embedding for covtau} so we focus on the first claim. 
Suppose first that $Y\to X$ is an object of $\Cov_X^\et$. We claim that the object $h^\sharp_{Y,\proet}$ of $\Sh(X_\proet)$ lies in the subcategory $\Loc(X_\proet)$. But, there exists an \'etale cover $\{U_i\to X\}$ such that $U_i$ is affinoid for all $i$, and such that $Y_{U_i}$ is an object of $\UFEt_{U_i}$ for all $i$. Since $h^\sharp_{Y,\proet}$ restricted to $U_i$ is $h^\sharp_{Y_{U_i},\proet}$ in $\Sh(X_\proet/U_i)=\Sh((U_i)_\proet)$ we know from Proposition \ref{prop:main-theorem-locally-constant-profinite-case} that $h^\sharp_{Y,\proet}$ restricted to each $U_i$ is locally constant. Thus, $h^\sharp_{Y,\proet}$ itself is locally constant. 

As our our functor is nothing but $\nu_X^\ast$ by Proposition \ref{prop:almost-representable-presheaf}, it is fully faithful by Proposition~\ref{classical corollary}. Given this, and the fact both the source $\Cov^\et_X$ and the target $\cat{Loc}(X_\proet)$ naturally form stacks on $X$ for the \'etale topology (the former by \cite[Corollary 3.1.9]{Warner}), it is enough to show essential surjectivity \'etale locally on $X$. Thus, it is enough to assume $X$ is affinoid and show that every sheaf of sets $\mc{F}$ of $X_\proet$ which becomes constant on a pro-finite \'etale cover of $X$ comes from an object of $\UFEt_X$. This is Proposition~\ref{prop:main-theorem-locally-constant-profinite-case}.
\end{proof}

\begin{cor}
    Let $(X, \ov x)$ be a connected pointed rigid $K$-space. Then, the stalk functor
    \begin{equation*}
        F_{\ov{x}}\colon \cat{ULoc}(X_\proet)\to\GSet{\pi_1^{\dJ,\et}(X,\ov{x})}
    \end{equation*}
    is an equivalence of categories, and consequently $(\cat{ULoc}(X_\proet),F_{\ov{x}})$ is a tame infinite Galois category with fundamental group $\pi_1^{\dJ,\et}(X,\ov{x})$.  
\end{cor}

\bibliographystyle{amsalpha}
\renewcommand{\MR}[1]{MR \href{http://www.ams.org/mathscinet-getitem?mr=#1}{#1}}
\providecommand{\bysame}{\leavevmode\hbox to3em{\hrulefill}\thinspace}
\providecommand{\MR}{\relax\ifhmode\unskip\space\fi MR }
\providecommand{\MRhref}[2]{%
  \href{http://www.ams.org/mathscinet-getitem?mr=#1}{#2}
}
\providecommand{\href}[2]{#2}

\end{document}